\documentclass[notitlepage]{scrartcl}
\usepackage{enumitem}
\usepackage{amsmath}
\usepackage{amssymb}
\usepackage{amsthm}
\usepackage{abstract}
\usepackage{xcolor}
\usepackage{comment}
\usepackage{hyperref}
\hypersetup{
    colorlinks=true,
    linkcolor=blue,
    filecolor=magenta,      
    urlcolor=cyan
    }
\newtheorem{theorem}{Theorem}
\newtheorem{thm}{Theorem}[section]

\newtheorem{lemma}[thm]{Lemma}
\newtheorem{claim}[thm]{Claim}

\newcommand{\footremember}[2]{%
    \footnote{#2}
    \newcounter{#1}
    \setcounter{#1}{\value{footnote}}%
}

\usepackage[margin=1in]{geometry}
\title{\vspace{-2em}Supercritical Site Percolation on the Hypercube: Small Components are Small}
\author{%
Sahar Diskin \footremember{alley}{School of Mathematical Sciences, Tel Aviv University, Tel Aviv 6997801, Israel. Email: sahardiskin@mail.tau.ac.il.}%
\and Michael Krivelevich \footremember{trailer}{School of Mathematical Sciences, Tel Aviv University, Tel Aviv 6997801, Israel. Email:
krivelev@tauex.tau.ac.il. Research supported in part by USA–Israel BSF grant 2018267.}%
}

\begin{document}
\maketitle
\vspace{-2em}
\begin{abstract}
    We consider supercritical site percolation on the $d$-dimensional hypercube $Q^d$. We show that typically all components in the percolated hypercube, besides the giant, are of size $O(d)$. This resolves a conjecture of Bollob\'as, Kohayakawa, and Łuczak from 1994.
\end{abstract}

\section{Introduction}
The $d$-dimensional hypercube $Q^d$ is the graph with the vertex set $V(Q^d)=\{0,1\}^d$, where two vertices are adjacent if they differ in exactly one coordinate. Throughout this paper, we denote by $n=2^d$ the order of the hypercube.

In bond percolation on $G$, one considers the subgraph $G_p$ obtained by including every edge independently with probability $p$. Erd\H{o}s and Spencer \cite{ES} conjectured that, similar to the classical $G(n,p)$ model, there is a phase transition at $p=\frac{1}{d}$ in $Q^d_p$: for $\epsilon>0$, when $p=\frac{1-\epsilon}{d}$, \textbf{whp} all components have size $O(d)$, and when $p=\frac{1+\epsilon}{d}$, there exists \textbf{whp} a unique giant component in $Q^d_p$ whose order is linear in $n$. Their conjecture was proved by Ajtai, Koml\'os, and Szemer\'edi \cite{AKS}. Bollob\'as, Kohayakawa, and Łuczak \cite{BKL2} improved upon that result, extending it to a wider range of $p$, allowing $\epsilon=o(1)$, and giving an asymptotic value of the order of the giant component of $Q^d_p$. Furthermore, they proved that the second largest component in $Q^d_p$ is typically of order $O(d)$.

In site percolation on $G$, one considers the induced subgraph $G[R]$, where $R$ is a random subset of vertices formed by including each vertex independently with probability $p$. In the setting of site percolation on the hypercube, Bollob\'as, Kohayakawa, and Łuczak \cite{BKL} proved the following:

\begin{thm}[Theorems 8, 9 of \cite{BKL}]\label{BKL}
\textit{Let $\epsilon>0$ be a small enough constant, and let $p=\frac{1+\epsilon}{d}$. Let $R$ be a random subset formed by including each vertex of $Q^d$ independently with probability $p$. Then, \textbf{whp}, there is a unique giant component of $Q^d[R]$, whose asymptotic size is $\frac{\left(2\epsilon-O(\epsilon^2)\right)n}{d}$. Furthermore, \textbf{whp} the size of the other components is at most $d^{10}$.}
\end{thm}

Motivated by the results in bond percolation on the hypercube, in the same paper \cite{BKL}, and also in their paper on bond percolation on the hypercube \cite{BKL2}, Bollobás, Kohayakawa, and Łuczak conjectured that \textbf{whp} all the components of $Q^d[R]$ besides the giant component are of asymptotic size at most $\gamma d$, where $\gamma$ is a constant depending only on $\epsilon$ (see Conjecture 11 of \cite{BKL}). We prove this conjecture:
\begin{theorem}\label{main-theorem}
Let $\epsilon>0$ be a small enough constant, and let $p=\frac{1+\epsilon}{d}$. Let $R$ be a random subset formed by including each vertex of $Q^d$ independently with probability $p$. Then, \textbf{whp}, all the components of $Q^d[R]$ besides the giant component are of size $O_{\epsilon}\left(d\right)$.
\end{theorem}

We note that in our proof we obtain an inverse polynomial dependency on $\epsilon$, of order $\frac{1}{\epsilon^5}$, in the hidden constant in $O_{\epsilon}(d)$.

The text is structured as follows. In Section $2$, we show that big components are 'everywhere dense' (see the exact definition in Lemma \ref{dense-giant-2}). In Section $3$, we exhibit several structures that \textbf{whp} do not appear in the site-percolated hypercube. In Section $4$, we prove Theorem \ref{main-theorem}.

Our notation is fairly standard. Throughout the rest of the text, unless specifically stated otherwise, we set $\epsilon$ to be a small enough constant, $p=\frac{1+\epsilon}{d}$, and $R$ to be a random subset of $V(Q^d)$ formed by including each vertex independently with probability $p$. Furthermore, we denote by $L_1$ the largest connected component of $Q^d[R]$. We denote by $N_G(S)$ the external neighbourhood of a set $S$ in the graph $G$. When $T$ is a subset of $V(G)$, we write $N_T(S)=N_G(S)\cap T$. We omit rounding signs for the sake of clarity of presentation.

\section{Big components are everywhere dense}
We begin with the following claim, which holds for any $d$-regular graph:
\begin{claim} \label{DFS-lemma}
\textit{Let $G=(V,E)$ be a $d$-regular graph on $n$ vertices. Let $\epsilon$ be a small enough constant, and let $p=\frac{1+\epsilon}{d}$. Form $R\subseteq V$ by including each vertex of $V$ independently with probability $p$. Then, \textbf{whp}, any connected component $S$ of $G[R]$ with $|S|=k>300\ln n$ has $|N_G(S)|\ge \frac{9kd}{10}$.}
\end{claim}
\begin{proof}
We run the Depth First Search (DFS) algorithm with $n$ random bits $X_i$, discovering the connected components while generating $G[R]$ (see \cite{site-percolation}). If there is a connected component $S$ of size $k$, then there is an interval of random bits whose length is $k+|N_G(S)|$ in the DFS where we receive $k$ positive answers. Thus, the probability of having a component violating the claim, whose discovery by the DFS starts with a given bit $X_i$ is at most:
\begin{align*}
    P\left[Bin\left(\frac{9kd}{10}+k,\frac{1+\epsilon}{d}\right)\ge k\right]
\end{align*}
Hence, by a Chernoff-type bound (see Appendix A in \cite{probablistic method}), the probability of having such a component is at most
\begin{align*}
    \exp\left(-\frac{k}{100}\right)\le o\left(\frac{1}{n^2}\right),
\end{align*}
where the inequality follows from our bound on $k$. We complete the proof with a union bound on the $<n$ possible values of $k$ and $< n$ possible starting points of the interval in the DFS.
\end{proof}

We further require the following simple claim:
\begin{claim} \label{seperating}
Let $S\subseteq V(Q^d)$ such that $|S|=k\le d$. Then, there exist pairwise disjoint subcubes, each of dimension at least $d-k+1$, such that every $v\in S$ is in exactly one of these subcubes. 
\end{claim}
\begin{proof}
We prove by induction on $k$. The case $k=1$ is trivial.

Assume that the statement holds for all $k'<k$. Choose any two vertices of $S$. There is at least one coordinate on which they do not agree. Consider the two pairwise disjoint (and complementary) subcubes of dimension $d-1$: one where we fix this coordinate to be $0$ and let the other coordinates vary, and the other where we fix this coordinate to be $1$ and let the other coordinates vary. Clearly, each of the two vertices is in exactly one of these subcubes. Since these subcubes are disjoint, no vertex from the other $k-2$ vertices is in both of them. We can then apply the induction hypothesis on each of these subcubes, giving rise to pairwise disjoint subcubes of dimension at least $d-1-(k-2)=d-k+1$, each having exactly one of the vertices of $S$.
\end{proof}

We conclude this section with a lemma, bounding the probability that a fixed small set of vertices has no vertex at Hamming distance $1$ from many vertices in large components or in their neighbourhoods. The proof here borrows several ideas used in \cite{AKS, BKL, EKK}.
\begin{lemma}\label{dense-giant-2}
Fix $S\subseteq V(Q^d)$ such that $|S|=k\le \frac{\epsilon d}{10}$. Then, the probability that every $v\in S$ has less than $\frac{\epsilon^2d}{40}$ neighbours (in $Q^d$) in components of $Q^d[R]$ whose size is at least $n^{1-\epsilon}$ or in their neighbourhoods in $Q^d$ is at most $\exp\left(-\frac{\epsilon^2dk}{40}\right)$.
\end{lemma}
\begin{proof}
Let $S=\{v_1,\cdots, v_k\}$. By Claim \ref{seperating}, we can consider pairwise disjoint subcubes of dimension at least $\left(1-\frac{\epsilon }{10}\right)d$, each containing exactly one of the vertices of $S$. We denote these subcubes by $Q_1,\cdots, Q_k$, where $v_i\in V(Q_i)$. 

For each $i$, consider $Q_i$, and assume, without loss of generality, that $v_i$ is the origin of $Q_i$. We then create $\frac{\epsilon d}{10}$ pairwise disjoint subcubes of $Q_i$ of dimension at least $\left(1-\frac{\epsilon}{5}\right)d$, each at distance $1$ from $v_i$, by fixing one of the first $\frac{\epsilon d}{10}$ coordinates of $Q_i$ to be $1$, the rest of the first $\frac{\epsilon d}{10}$ coordinates to be $0$, and letting the other coordinates vary. Denote these subcubes by $Q_i(1), \cdots, Q_i\left(\frac{\epsilon d}{10}\right)$, and the vertex at distance $1$ from $v_i$ in $Q_i(j)$ by $v_i(j)$. We denote by $n'$ the order of each $Q_i(j)$, and note that $n'\ge2^{\left(1-\frac{\epsilon}{5}\right)d}=n^{1-\frac{\epsilon}{5}}$.

Fix $i$. Observe that $p$ is still supercritical at every $Q_i(j)$ since $(1+\epsilon)\left(1-\frac{\epsilon}{5}\right)\ge 1+\frac{3\epsilon}{5}$. Denote by $L_1(i,j)$ the largest component of $Q_i(j)[R]$. Then, by Theorem \ref{BKL}, \textbf{whp} $\big|L_1(i,j)\big|\ge \frac{7\epsilon n'}{6d}.$ Furthermore, by Claim \ref{DFS-lemma}, \textbf{whp}
$\bigg|N_{Q_i(j)}\left(L_1(i,j)\right)\bigg|\ge \frac{21\epsilon n'}{20}.$ Thus, setting $\mathcal{A}_{(i,j)}$ to be the event that $\bigg|L_1(i,j)\cup N_{Q_i(j)}\left(L_1(i,j)\right)\bigg|\ge\frac{21\epsilon n'}{20}$, we have that $P\left(\mathcal{A}_{(i,j)}\right)=1-o(1)$. Now, let $\mathcal{B}_{(i,j)}$ be the event that $v_i(j)\in L_1(i,j)\cup N_{Q_i(j)}\left(L_1(i,j)\right)$. Since the hypercube is transitive, we have that $P\left(\mathcal{B}_{(i,j)}|\mathcal{A}_{(i,j)}\right)\ge \frac{21\epsilon n'/20}{n'}=\frac{21\epsilon}{20}$. Therefore, $P\left(\mathcal{B}_{(i,j)}\cap \mathcal{A}_{(i,j)}\right)\ge \left(1-o(1)\right)\frac{21 \epsilon}{20}\ge \epsilon$. Since the subcubes $Q_i(j)$ are pairwise disjoint, the events $B_{(i,j)}\cap \mathcal{A}_{(i,j)}$ are independent for different $j$. Thus, by a typical Chernoff-type bound, with probability at least $1-\exp\left(-\frac{\epsilon^2d}{40}\right)$, at least $\frac{\epsilon^2d}{40}$ of the $v_i(j)$ are in $L_1(i,j)\cup N_{Q_i(j)}\left(L_1(i,j)\right)$ with $\bigg|L_1(i,j)\cup N_{Q_i(j)}\left(L_1(i,j)\right)\bigg|\ge\frac{21\epsilon n'}{20}$. Thus, with the same probability, $v_i$ is at distance $1$ from at least $\frac{\epsilon^2d}{40}$ vertices in components whose size is at least $\frac{\epsilon n'}{d}\ge \frac{\epsilon n^{1-\frac{\epsilon}{5}}}{d}\ge n^{1-\epsilon}$ or in their neighbourhoods.

Since $v_i$'s are in pairwise disjoint subcubes, these events are also independent for each $v_i$. Hence, the probability that none of the $v_i$ are at distance $1$ from at least $\frac{\epsilon^2d}{40}$ vertices in components whose size is at least $n^{1-\epsilon}$ or in their neighbourhoods is at most $\exp\left(-\frac{\epsilon^2dk}{40}\right)$.
\end{proof}

\section{Unlikely structures in the percolated hypercube}

Denote by $N^2_G(v)$ the set of vertices in $G$ at distance exactly $2$ from $v$. The following lemma shows that, typically, there are no large sections of a sphere of radius $2$ in $Q^d[R]$.
\begin{lemma}\label{no-sphere}
\textbf{Whp}, there is no $v\in Q^d$ such that $\big|N^2_{Q^d}(v)\cap R\big|\ge 2d$.
\end{lemma}
\begin{proof}
We have $n$ ways to choose $v$, and $\binom{\binom{d}{2}}{2d}$ ways to choose a subset of $2d$ vertices from $N^2_{Q^d}(v)$. We include them in $R$ with probability at most $p^{2d}$. Hence, by the union bound, the probability of violating the lemma is at most:
\begin{align*}
    2^d\binom{\binom{d}{2}}{2d}\left(\frac{1+\epsilon}{d}\right)^{2d}\le 2^d\left(\frac{ed}{4}\cdot \frac{1+\epsilon}{d}\right)^{2d}\le \left(\frac{14}{15}\right)^{d}=o(1).
\end{align*}
\end{proof}

From Lemma \ref{no-sphere}, we are able to derive the following lemma:
\begin{lemma} \label{big-neighbourhood}
\textbf{Whp} there are no $S\subseteq R$ and $W\subseteq V(Q^d)$ disjoint from $S$ such that $|W|\le\frac{\epsilon^4d|S|}{9\cdot 200^2}$ and every $v\in S$ has $d_{Q^d}(v, W)\ge \frac{\epsilon^2d}{200}$.
\end{lemma}
\begin{proof}
Assume the contrary. By our assumption on $S$, there are at least $\frac{\epsilon^2 d|S|}{200}$ edges between $S$ and $W$. Thus, the average degree from $W$ to $S$ is at least $\frac{\epsilon^2 d|S|}{200|W|}$. Now, let us count the number of cherries with the vertex of degree $2$ in $W$. By Jensen's inequality, we have at least $\binom{\epsilon^2 d|S|/200|W|}{2}|W|\ge \frac{\epsilon^4d^2|S|^2}{4\cdot 200^2|W|}\ge \frac{9d|S|}{4}$ such cherries, where we used our assumption on $|W|$. Thus, by the pigeonhole principle, there is a vertex $v\in S$ that is in $\frac{2}{|S|}\cdot\frac{9d|S|}{4}=\frac{9d}{2}$ cherries. Since every pair of vertices in $S$ is connected by at most two paths of length $2$ in $Q^d$, we obtain that $v$ is at Hamming distance $2$ from at least $\frac{\frac{9d}{2}}{2}=\frac{9d}{4}>2d$ vertices in $S\subseteq R$. On the other hand, by Lemma \ref{no-sphere}, \textbf{whp} there is no $v\in Q^d$ such that $|N_{Q^d}^2(v)\cap R|\ge 2d$, completing the proof.
\end{proof}

The next lemma bounds the number of subtrees of a given order in a $d$-regular graph.
\begin{lemma} \label{trees}
\textit{Let $t_k(G)$ denote the number of $k$-vertex trees contained in a $d$-regular graph $G$ on $n$ vertices. Then,
$$t_k(G)\le n(ed)^{k-1}.$$}
\end{lemma}
This follows directly from Lemma 2.1 of \cite{trees}.

We are now ready to state and prove the final lemma of this section:
\begin{lemma}\label{no-squid}
Let $C>0$ be a constant. Then, \textbf{whp}, there is no $S\subseteq V(Q^d)$, such that $|S|\le Cd$, $S$ is connected in $Q^d$, and at least $\frac{\epsilon d}{10}$ vertices $v\in S$ have that $\big|N_{L_1\cup N_{Q^d}(L_1)}(v)\big|<\frac{\epsilon^2d}{40}$.
\end{lemma}
\begin{proof}
By Theorem \ref{BKL}, \textbf{whp} there is a unique giant component whose size is larger than $d^{10}$, which we denote by $L_1$. Thus, it suffices to show that \textbf{whp} there is no such $S$, where $\frac{\epsilon d}{10}$ of its vertices have less than $\frac{\epsilon^2d}{40}$ vertices at Hamming distance $1$ from components of size at least $n^{1-\epsilon}$ or their neighbourhoods, since \textbf{whp} these components and their neighbourhoods are in fact $L_1\cup N_{Q^d}(L_1)$. 

Since $Q^d$ is connected, we can consider connected sets of size exactly $Cd$. By Lemma \ref{trees}, we have $n(ed)^{Cd}$ ways to choose $S$. We have $\binom{Cd}{\frac{\epsilon d}{10}}$ ways to choose the vertices in $S$ which do not have at least $\frac{\epsilon^2d}{40}$ vertices at Hamming distance $1$ from components of size at least $n^{1-\epsilon}$ or their neighbourhoods. By Lemma \ref{dense-giant-2}, the probability that no vertex in a given set of $\frac{\epsilon d}{10}$ vertices in $S$ has at least $\frac{\epsilon^2d}{40}$ vertices at Hamming distance $1$ from components of size at least $n^{1-\epsilon}$ or their neighbourhoods is at most $\exp\left(-\frac{\epsilon^4d^2}{400}\right)$. Hence, the probability of the event violating the statement of the lemma is at most:
\begin{align*}
    n(ed)^{Cd}\binom{Cd}{\frac{\epsilon d}{10}}\exp\left(-\frac{\epsilon^4d^2}{400}\right)\le n\left((ed)^C\left(\frac{10eC}{\epsilon}\right)^{\frac{\epsilon}{10}}\exp\left(-\frac{\epsilon^4d}{400}\right)\right)^d=o(1).
\end{align*} 
\end{proof}

\section{Proof of Theorem \ref{main-theorem}}
Let $p_1=\frac{1+\epsilon/2}{d}$. Form $R_1$ by including each vertex $v\in V(Q^d)$ independently with probability $p_1$. By Theorem \ref{BKL}, \textbf{whp} there is a unique giant component, denote it by $L_1'$. We can thus split the vertices of the hypercube into the following three disjoint sets: $T=L_1'\cup N_{Q^d}(L_1')$, $M$ is the set of vertices outside $T$ with at least $\frac{\epsilon^2d}{200}$ neighbours in $T$, and $S=V\left(Q^d\right)\setminus\left(T\cup M\right)$.

Let $p_2=\frac{\epsilon}{2d-2-\epsilon}$. Form $R_2$ by including each vertex $v\in V(Q^d)$ independently with probability $p_2$. Note that since $1-p=(1-p_1)(1-p_2)$, the random set $R$ has the same distribution as $R_1\cup R_2$. Thus, with a slight abuse of notation, we write $R=R_1\cup R_2$. We begin by performing the second exposure (i.e. by generating $R_2$) on $S\cup M$, and only afterwards on $T$. Note that once we show that a connected set has a neighbour in $T\cap R_2$, it means that it merges into $L_1'$ (whereas by Theorem \ref{BKL}, \textbf{whp} $L_1'\subseteq L_1$).

Let $C_1$ be a positive constant, possibly depending on $\epsilon$. We first show that \textbf{whp} if there is a connected component $B$ in $(S\cup M)\cap R$ such that $\big|B\cap M\big|\ge C_1d$, it merges with $L_1'$ after the second exposure on $T$. By construction, every $v\in M$ has at least $\frac{\epsilon^2d}{200}$ neighbours in $T$, and $T\cap M=\emptyset$. Thus, by Lemma \ref{big-neighbourhood}, we have that \textbf{whp} $\big|N_T(B\cap M)\big|\ge \frac{C_1\epsilon^4d^2}{9\cdot 200^2}$. Hence,
\begin{align*}
    P\left[\big|N_T(B)\cap R_2\big|=0\right]&\le \left(1-\frac{\epsilon}{2d}\right)^{\frac{C_1\epsilon^4d^2}{9\cdot200^2}}\le \exp\left(-\frac{C_1\epsilon^5d}{18\cdot200^2}\right)=o(1/n),
\end{align*}
by choosing $C_1\ge\frac{18\cdot200^2}{\epsilon^5}$. We conclude with the union bound over the $<n$ connected components in $(S\cup M)\cap R$.

Thus, we are left to deal with connected components $B\subseteq(S\cup M)\cap R$ of size at least $2C_1d$ with less than $C_1d$ vertices in $B\cap M$. Taking a connected subset $B'\subseteq B$ of size exactly $Cd:=2C_1d$, we have that $|B'\cap S|\ge |B'|-|B'\cap M|\ge C_1d\ge \frac{\epsilon d}{20}$. But by Lemma \ref{no-squid}, \textbf{whp} there is no connected set of size $Cd$ in $Q^d$ with at least $\frac{\frac{\epsilon}{2} d}{10}=\frac{\epsilon d}{20}$ vertices which do not have at least $\frac{\left(\frac{\epsilon}{2}\right)^2d}{40}\ge \frac{\epsilon^2d}{200}$ vertices at Hamming distance $1$ from $L_1'\cup N_{Q^d}(L_1')$. Recalling that $S, M,T$ and $L_1'$ were formed according to the supercritical percolation with $p_1=\frac{1+\epsilon/2}{d}$, this means that \textbf{whp} there is no such connected set $B'$. All in all, \textbf{whp}, there is no connected component $B$ in $Q^d[R]$ of size at least $Cd$ outside of $L_1$. \qedsymbol

\paragraph{Acknowledgements.} The first author wishes to thank Arnon Chor and Dor Elboim for fruitful discussions.

\end{document}